\documentclass[12pt]{amsart}
\usepackage{amsmath, amsthm, amscd, amsfonts, amssymb, graphicx, color,float,pgf,tikz, mathrsfs}
\usepackage{amssymb,fontenc}
\usepackage{latexsym,wasysym,mathrsfs}
\usepackage{hyperref}
\usepackage{booktabs} 
\usepackage{array}
\usepackage[utf8]{inputenc}
\usepackage[T1]{fontenc}
\usepackage{booktabs}
\usepackage{geometry}
\usepackage{subcaption}
\usepackage{soul}

\newtheorem{theorem}{Theorem}[section]
\newtheorem{lemma}[theorem]{Lemma}

\newtheorem{corollary}[theorem]{Corollary}

\theoremstyle{definition}
\newtheorem{definition}[theorem]{Definition}

\newtheorem{remark}[theorem]{Remark}
\newcommand{\F}{{\mathcal F}}

\newcommand{\MF}{{\mathcal MF}}
\newcommand{\MB}{{\mathcal MB}}
\newcommand{\TB}{{\mathcal TB}}
\newcommand{\GB}{{\mathcal GB}}

\def\stirling2#1#2{\genfrac{\{}{\}}{0pt}{}{#1}{#2}}

\title{Prime and Touchard Congruences of Mixed-Type Bell Numbers}
\author{Daniel Yaqubi$^{*}$}
\address{\bf $^*$Department of Computer science, University of Torbat-e Jam, Torbat-e Jam, Iran.}
\email{yaqubi@tjamcaas.ac.ir, or daniel\_yaqubi@yahoo.es}
\author{Madjid Mirzavaziri$^{**}$}
\address{\bf Department of Pure Mathematics, Faculty of Mathematical Sciences,  Ferdowsi University of Mashhad, P. O. Box 1159-91775, Mashhad, Iran.}
\email{mirzavaziri@gmail.com}

\subjclass[2020]{Primary: 11B73; Secondary: 05A18.}
\keywords{Stirling numbers of the second kind; Touchard congruences; Fubini numbers; partition of a multi-set; Mixed type Bell numbers.}
\begin{document}

\begin{abstract}
This paper establishes a comprehensive combinatorial and arithmetic framework for mixed Stirling and mixed Bell numbers, bridging partition structures, Touchard polynomials, and prime-power congruences. Furthermore, we develop to $p$-adic  valuation theory, proving prime-power Touchard congruences and higher-order modulo-$p^2$ refinements that generalize classical arithmetic properties of combinatorial sequences.
\end{abstract}
\maketitle
\section{Introduction}
The classical problem of partitioning an $n$-element set into $k$ non-empty blocks, enumerated by the Stirling numbers of the second kind ${n \brace k}$, naturally generalizes to multiset partitions. Let $\mathbf{b}$ and $\mathbf{c}$ denote the multiplicity vectors of balls and cells, respectively. The \emph{mixed (multiset) Stirling numbers of the second kind}, $\mathcal{S}(\mathbf{b};\mathbf{c})$, count the admissible partitions of the multiset $\mathbf{b}$ into cells $\mathbf{c}$ \cite{Yaqubi2016, YaqubiPreprint}. Summing over all admissible non-zero cell-multiplicity vectors defines the \emph{mixed Bell number}, $\mathcal{B}(\mathbf{b}) = \sum_{\mathbf{c}} \mathcal{S}(\mathbf{b};\mathbf{c})$, which recovers the classical Bell numbers $B_n$ when $\mathbf{b} = \mathbf{1}^n$.

Exploiting the exponential generating function (EGF) for a fixed cell-multiplicity vector $\mathbf{c} = (c_1, \ldots, c_k)$, given by $\sum_{n \geq 0} \mathcal{S}(n;\mathbf{c}) \frac{x^n}{n!} = \prod_{i=1}^{k} \frac{(e^x-1)^{c_i}}{c_i!}$ \cite{Yaqubi2019}, we introduced three novel Bell-type sequences arising from specialized mixed occupancy structures \cite{YaqubiPreprint}.
\begin{definition}
For a fixed integer $k \geq 1$, the \emph{general mixed Bell number of order $k$}, $\mathcal{GB}_n^{(k)}$, is the sum of mixed Stirling numbers over all non-negative vectors $\mathbf{c} = (c_1, \ldots, c_k)$. Its EGF is given by
\begin{equation}
    \sum_{n \geq 0} \mathcal{GB}_n^{(k)} \frac{x^n}{n!} = \exp\bigl(k(e^x-1)\bigr).
    \label{eq:general-mixed-bell-egf}
\end{equation}
\end{definition}

\begin{definition}
The \emph{mixed Bell number of order $k$}, $\mathcal{MB}_{n,k}=\sum_{m\geq1} \mathcal{S}\bigl(n;(m,1^k)\bigr)$, restricts the cells to one unlabeled cell of arbitrary multiplicity $m \geq 1$ and $k$ labeled cells of unit multiplicity. Summing over all $n \geq 1$ yields the EGF
\begin{equation}
    \sum_{n \geq 0} \mathcal{MB}_{n,k} \frac{x^n}{n!} = (e^x-1)^k \left(e^{e^x-1}-1\right).
    \label{eq:MB-order-egf}
\end{equation}
\end{definition}
\begin{definition}
The \emph{total mixed Bell number}, $\mathcal{TB}_n = \sum_{k \geq 1} \mathcal{MB}_{n,k}$, cumulatively sums the order-$k$ numbers over all valid $k$. Its EGF is
\begin{equation}
    \sum_{n \geq 0} \mathcal{TB}_n \frac{x^n}{n!} = \left(e^{e^x-1}-1\right) \frac{e^x-1}{2-e^x}.
    \label{eq:TMB-egf}
\end{equation}
\end{definition}
\section{Combinatorial Interpretations and Basic Properties}
\label{sec:combinatorial-connections}

We establish precise combinatorial interpretations for the three mixed Bell families, connecting them to classical sequences such as Touchard polynomials, Fubini numbers, and distinguished-block partitions. For the general mixed Bell numbers, restricting to $k$ labeled cell types with multiplicities $0$ or $1$ yields the exponential generating function (EGF). This coincides exactly with the EGF of the Touchard polynomials $T_n(y)=\sum_{j=0}^n{n\brace j}y^j$ evaluated at $y=k$. Thus
\begin{equation}
\GB^{(k)}_n=T_n(k)=\sum_{j=0}^n{n\brace j}k^j,
\end{equation}
which enumerates the classical $k$-colored set partitions of $[n]$. Specializing to $k=1$ recovers the ordinary Bell numbers $B_n$ (\href{https://oeis.org/A000110}{OEIS A000110}), while $k=2$ corresponds to \href{https://oeis.org/A001861}{OEIS A001861}.

For the order-$k$ mixed Bell numbers, the EGF factorization \ref{eq:MB-order-egf} shows that $MB_{n,k}$ counts set partitions of $[n]$ into exactly $k$ distinguished, labeled blocks and \emph{at least one} additional ordinary, unlabeled block.
reveals that $\mathcal{MB}_{n,k}$ counts set partitions of $[n]$ into exactly $k$ distinguished, labeled blocks and at least one ordinary, unlabeled block. Let $D(n,k)$ denote the number of set partitions of $[n]$ with $k$ distinguished blocks (\href{https://oeis.org/A049020}{OEIS A049020}). Subtracting the $k!{n\brace k}$ configurations lacking ordinary blocks gives the convolution
\begin{equation}
    \mathcal{MB}_{n,k} = k! \sum_{j=k}^{n-1} \binom{n}{j} {j\brace k} B_{n-j} = k!\left(D(n,k)-{n\brace k}\right).
\end{equation}
Also, the specialization $\MB_{n,1}=B_{n+1}-B_n-1=\mathcal{MB}_{n,1} = \text{\href{https://oeis.org/A005493}{OEIS A005493}}(n-1)-1.$
Summing over $k\ge1$ yields the total mixed Bell numbers $\TB_n$, whose EGF \ref{eq:TMB-egf} factors into the EGFs of nonempty preferential arrangements (Fubini numbers $\F_n$, \href{https://oeis.org/A000670}{OEIS A000670}) and nonempty ordinary set partitions. Combinatorially, $\mathcal{TB}_n$ enumerates the ways to place an ordered set partition on a nonempty proper subset $S \subset [n]$, and an ordinary partition on $[n]\setminus S$
\begin{equation}
\label{Dnk}
    \mathcal{TB}_n = \sum_{j=1}^{n-1} \binom{n}{j} \F_j B_{n-j} = \sum_{k=1}^{n-1} k!\left(D(n,k)-{n\brace k}\right).
\end{equation}
These sequences obey a unified system of recurrences, derivable directly from their defining EGFs.
\begin{theorem}
For $n \ge 0$, the mixed Bell families satisfy the following relations:
\begin{enumerate}
    \item $\displaystyle \mathcal{GB}_{n+1}^{(k)} = k \sum_{j=0}^{n} \binom{n}{j} \mathcal{GB}_j^{(k)} = k \left( \mathcal{GB}_n^{(k)} + \frac{d}{dk}\mathcal{GB}_n^{(k)} \right)$.
    \item $\displaystyle \MB_{n+1,k}=\sum_{j=0}^n\binom{n}{j}\Big(k\,\MB_{j,k-1}+\MB_{j,k}+k!{j\brace k}\Big).$
\item $\displaystyle \TB_{n+1}=\sum_{j=1}^{n}\binom{n+1}{j}\F_jB_{n+1-j}.$
\end{enumerate}
\end{theorem}

\begin{proof}
 For (1), differentiating $\Phi_k(x) = \exp(k(e^x-1))$ with respect to $k$ provides the supplementary relation $\frac{\partial \Phi_k}{\partial k} = (e^x-1)\Phi_k(x)$, which translates algebraically to the stated derivative form. For (2), differentiating the EGF $\Phi_k(x)=(e^x-1)^k(e^{e^x-1}-1)$ directly gives
\[
\Phi_k'(x)=e^x\Big[k\,\Phi_{k-1}(x)+\Phi_k(x)+(e^x-1)^k\Big],
\]
and extracting coefficients (Cauchy product with $e^x$) gives exactly the corrected statement above. Part (3) follows from \ref{Dnk} by the index shift $n\mapsto n+1$.
\end{proof}
\subsection{Initial Values and Numerical Tables}
\label{subsec:initial-values}

We present initial values for the general mixed Bell polynomials, illustrating their rapid growth and connections to classical sequences.

\begin{table}[h]
\centering
\caption{General mixed Bell polynomials $\mathcal{GB}_n^{(k)}$}
\label{tab:GB_polynomials}
\renewcommand{\arraystretch}{1.3}
\begin{tabular}{c|l}
\hline
$n$ & $\mathcal{GB}_n^{(k)} = \sum_{j=0}^{n} {n \brace j} k^j$ \\ \hline
0 & $1$ \\
1 & $k$ \\
2 & $k^2 + k$ \\
3 & $k^3 + 3k^2 + k$ \\
4 & $k^4 + 6k^3 + 7k^2 + k$ \\
5 & $k^5 + 10k^4 + 25k^3 + 15k^2 + k$ \\
6 & $k^6 + 15k^5 + 65k^4 + 90k^3 + 31k^2 + k$ \\
7 & $k^7 + 21k^6 + 140k^5 + 350k^4 + 301k^3 + 63k^2 + k$ \\
8 & $k^8 + 28k^7 + 266k^6 + 1050k^5 + 1701k^4 + 966k^3 + 127k^2 + k$ \\ \hline
\end{tabular}
\end{table}
The polynomials $\mathcal{GB}_n^{(k)}$ coincide with the Touchard polynomials, with coefficients given by the Stirling numbers of the second kind (\href{https://oeis.org/A008277}{OEIS A008277}) ). Specializations, for
 $k=1$, is ordinary Bell numbers $B_n$ (\href{https://oeis.org/A000110}{OEIS A000110}): $1, 1, 2, 5, 15, 52, 203, 877, \dots$ and for $k=2$ is sequence $1, 2, 6, 22, 94, 454, \dots$ (\href{https://oeis.org/A000629}{OEIS A000629} ) for two-colored set partitions.

\begin{table}[h]
\centering
\caption{Corrected initial values of the mixed Bell numbers of order $k$ ($\MB_{n,k}$)}
\label{tab:MB-values}
\begin{tabular}{c|cccccc}
\toprule
$n\backslash k$ & 0 & 1 & 2 & 3 & 4 & 5\\
\midrule
0 & 1 & 0 & 0 & 0 & 0 & 0\\
1 & 1 & 0 & 0 & 0 & 0 & 0\\
2 & 2 & 2 & 0 & 0 & 0 & 0\\
3 & 5 & 9 & 6 & 0 & 0 & 0\\
4 & 15 & 36 & 48 & 24 & 0 & 0\\
5 & 52 & 150 & 290 & 300 & 120 & 0\\
6 & 203 & 673 & 1650 & 2580 & 2160 & 720\\
\bottomrule
\end{tabular}
\end{table}

The entry $\MB_{n,k}$ of table \ref{tab:MB-values} counts mixed partitions with $k$ labeled cells of multiplicity one and a remaining, nonempty, unlabeled block component. When $k=0$, the sequence reduces to the ordinary Bell numbers $B_n$: $1,1,2,5,15,52,203,\dots$ (note the $k=0$ column is $B_n$ itself, starting at $B_0=1$). For $k=1$, $\MB_{n,1}=B_{n+1}-B_n-1$ for $n\ge1$, beginning $0,2,9,36,150,673,\dots$ More generally,
\begin{equation}
\MB_{n,k}=k!\sum_{j=k}^{n-1}\binom{n}{j}{j\brace k}B_{n-j}.
\end{equation}

\begin{table}[h]
\centering
\caption{Initial values of the total mixed Bell numbers $\mathcal{TB}_n$}
\label{tab:TB-values}
\renewcommand{\arraystretch}{1.3}
\begin{tabular}{cc|cc}
\hline
$n$ & $\mathcal{TB}_n$ & $n$ & $\mathcal{TB}_n$ \\ \hline
$0$ & $0$ & $6$ & $7,783$ \\
$1$ & $0$ & $7$ & $80,150$ \\
$2$ & $2$ & $8$ & $932,746$ \\
$3$ & $15$ & $9$ & $12,152,013$ \\
$4$ & $108$ & $10$ & $175,550,014$ \\
$5$ & $860$ & & \\ \hline
\end{tabular}
\end{table}

The total mixed Bell numbers, defined by $\mathcal{TB}_n = \sum_{k \ge 1} \mathcal{MB}_{n,k}$, enumerate mixed partitions containing at least one labeled cell. The sequence begins as $0, 0, 2, 15, 108, 860, 7783, \dots$. 

A natural link to the Fubini numbers $\F_n$ (\href{https://oeis.org/A000670}{OEIS A000670}) arises through the weighted Stirling representation
\begin{equation}
    \mathcal{TB}_n = \sum_{r=2}^{n} {n \brace r} W_r,
    \label{eq:TB-stirling-form}
\end{equation}
where the weight is given by $W_r = \sum_{k=1}^{r-1} \binom{r}{k} k! = r! \sum_{j=1}^{r-1} \frac{1}{j!}$. While Fubini numbers order all $r$ blocks (weight $r!$), the total mixed construction distinguishes only a proper nonempty subset of blocks, providing a natural intermediate family between Bell- and Fubini-type enumerations.

The introduction of the three families of mixed partition sequences naturally motivates the definition of mixed Fubini numbers. These numbers may be viewed as a generalization of the classical Fubini numbers, which arise as a specific specialization within this broader framework.
\begin{definition}
\label{def:mixed-fubini}
For integers $n\ge0$ and $m\ge1$, the \emph{mixed Fubini number of type $(n,m)$} is
\begin{equation}
\MF_{n,m}=\sum_{k=0}^{n-m}\mathcal{S}\big(n;(m,1^k)\big),
\end{equation}
where $\mathcal{S}(n;(m,1^k))$ denotes the mixed Stirling number for one unlabeled component of $m$ blocks and $k$ additional, pairwise labeled, blocks of multiplicity one.
\end{definition}
Using $\mathcal{S}(n;(m,1^k))=\frac{(m+k)!}{m!}{n\brace m+k}$, the closed-form representations of mixed fubini numbers is
\begin{align}
\MF_{n,m}&=\sum_{k=0}^{n-m}\frac{(m+k)!}{m!}{n\brace m+k}=\sum_{j=m}^n\frac{j!}{m!}{n\brace j},\\
&=\sum_{j=0}^n\binom{n}{j}{j\brace m}\F_{n-j},
\end{align}
wherethe Cauchy product of the Stirling EGF $(e^x-1)^m/m!$ with the Fubini EGF $1/(2-e^x)$. The initial values of $MF_{n,m}$ are displayed in Table 5.

\begin{table}[h]
\centering
\caption{Initial values of the mixed Fubini numbers $\MF_{n,m}$}
\begin{tabular}{c|ccccc}
\toprule
$n\backslash m$ & 1 & 2 & 3 & 4 & 5\\
\midrule
1 & 1 & 0 & 0 & 0 & 0\\
2 & 3 & 1 & 0 & 0 & 0\\
3 & 13 & 6 & 1 & 0 & 0\\
4 & 75 & 37 & 10 & 1 & 0\\
5 & 541 & 270 & 85 & 15 & 1\\
\bottomrule
\end{tabular}
\end{table}

While classical Fubini numbers (ordered Bell numbers, \href{https://oeis.org/A000670}{OEIS A000670}) are given by $\F_n = \sum_{j=0}^{n} j! {n \brace j}$, the mixed Fubini numbers refine this construction by ordering only the labeled component. Consequently, $\mathcal{MF}_{n,m}$ provides a natural bridge connecting mixed Bell-type and Fubini-type enumerations.
\begin{theorem}
\label{thm:MF-EGF-symbolic}
For $m\ge1$, let $\MF_{n,m}$ denote the number of labeled structures consisting of a set of exactly $m$ nonempty unordered blocks together with an ordered sequence of an arbitrary number of additional nonempty blocks. Then its EGF is given by
\begin{equation}
\MF_m(x)=\sum_{n\ge0}\MF_{n,m}\frac{x^n}{n!}=\frac{(e^x-1)^m}{m!(2-e^x)}.
\end{equation}
\end{theorem}
\begin{proof}
We derive the EGF using the symbolic method for labeled combinatorial classes \cite{FlajoletSedgewick}. Let $Z$ denote the atomic labeled class of size $1$. The class of nonempty sets of labeled atoms, representing a single nonempty block, is $\mathcal{A}=\mathrm{SET}_{\ge1}(Z)$, with EGF $A(x)=e^x-1$. The mixed Fubini structure decomposes into two independent components: a set of exactly $m$ nonempty blocks, $\mathrm{SET}_m(\mathcal{A})$, with EGF $A(x)^m/m!=(e^x-1)^m/m!$; and an ordered sequence of an arbitrary number of nonempty blocks, $\mathrm{SEQ}(\mathcal{A})$, with EGF $\sum_{k\ge0}A(x)^k=1/(1-(e^x-1))=1/(2-e^x)$. As a labeled product, $\MF_m(x)=\frac{(e^x-1)^m}{m!}\cdot\frac{1}{2-e^x}$.
\end{proof}
\begin{theorem}
\label{thm:fubini_system}
For $m \ge 0$, let $\F_m(x) = \sum_{n \ge 0} \mathcal{MF}_{n,m} \frac{x^n}{n!} = \frac{(e^x-1)^m}{m!(2-e^x)}$ be the exponential generating function of the mixed Fubini numbers. For $m \ge 1$, the following relations hold
\begin{enumerate}
    \item Differential relation 
    \begin{equation}
        \F_m'(x) = e^x \F_{m-1}(x) + \frac{e^x}{2-e^x} \F_m(x).
        \label{eq:fubini_diff_sys}
    \end{equation}
    \item Convolution recurrence 
    \begin{equation}
        \mathcal{MF}_{n+1,m} = \sum_{i=0}^{n} \binom{n}{i} \mathcal{MF}_{i,m-1} + 2 \sum_{i=0}^{n} \binom{n}{i} \F_i \mathcal{MF}_{n-i,m} - \mathcal{MF}_{n,m},
        \label{eq:fubini_conv_correct}
    \end{equation}
    where $\F_i = \mathcal{MF}_{i,0}$ denotes the classical Fubini number.
    \item Algebraic recurrence 
    \begin{equation}
        \sum_{i=0}^{n} \binom{n}{i} \mathcal{MF}_{i,m} = 2 \mathcal{MF}_{n,m} - {n \brace m}.
        \label{eq:fubini_algebraic_recurrence}
    \end{equation}
\end{enumerate}
Boundary values are given by $\mathcal{MF}_{n,0} = \F_n$, $\mathcal{MF}_{0,0} = 1$, and $\mathcal{MF}_{n,m} = 0$ for $m > n$.
\end{theorem}
\begin{proof}
Differentiating $\F_m(x) = \frac{(e^x-1)^m}{m!(2-e^x)}$  \eqref{eq:fubini_diff_sys}gives
\[
   \F_m'(x) = \frac{e^x(e^x-1)^{m-1}}{(m-1)!(2-e^x)} + \frac{e^x(e^x-1)^m}{m!(2-e^x)^2} = e^x \F_{m-1}(x) + \frac{e^x}{2-e^x} \F_m(x).
\]
Noting that $\frac{e^x}{2-e^x} = \frac{2}{2-e^x} - 1 = 2\F_0(x) - 1$, the differential equation rewrites as $\F_m'(x) = e^x \F_{m-1}(x) + (2F_0(x)-1)\F_m(x)$. Extracting coefficients of $x^n/n!$ via the EGF product rule establishes \eqref{eq:fubini_conv_correct}.
Finally, multiplying both sides of the definition by $2-e^x$ gives $(2-e^x)\F_m(x) = \frac{(e^x-1)^m}{m!} = \sum_{n \ge 0} {n \brace m} \frac{x^n}{n!}$. Equating coefficients using the Cauchy product for $e^x \F_m(x)$ yields the algebraic recurrence \eqref{eq:fubini_algebraic_recurrence}.
\end{proof}
\begin{corollary}
\label{cor:MF-Stirling}
For all $n, m \ge 0$, 
\begin{equation}
    2\mathcal{MF}_{n,m} - \sum_{i=0}^{n} \binom{n}{i}\mathcal{MF}_{i,m} = {n \brace m}.
    \label{eq:MF-Stirling}
\end{equation}
\end{corollary}
\begin{proof}
Multiplying the identity $(2-e^x)\F_m(x) = \frac{(e^x-1)^m}{m!} = \sum_{n \ge 0} {n \brace m} \frac{x^n}{n!}$ and extracting coefficients via the Cauchy product yields \eqref{eq:MF-Stirling}.
\end{proof}
\begin{corollary}
\label{cor:MF-binomial}
For all $n \ge m \ge 0$,
\begin{equation}
\MF_{n,m}=\sum_{j=0}^n\binom{n}{j}{j\brace m}\F_{n-j},
 \label{eq:MF-binomial}
\end{equation}
where $F_n$ denotes the classical Fubini numbers.
\end{corollary}

\begin{proof}
Expressing $\F_m(x) = \frac{(e^x-1)^m}{m!} \cdot \frac{1}{2-e^x}$ as the product of the EGFs for Stirling numbers of the second kind and classical Fubini numbers, the formula \eqref{eq:MF-binomial} follows directly from the Cauchy product.
\end{proof}
The mixed Fubini numbers $\mathcal{MF}_{n,m}$ bridge unordered and ordered partition structures, where $m$ counts the unlabeled blocks and the remaining blocks form a linear sequence.
\begin{remark}
\label{rem:MF-classical}
For $m=0$, the EGF reduces to $\F_0(x) = (2-e^x)^{-1}$, yielding the classical Fubini numbers $\mathcal{MF}_{n,0} = \F_n$.
\end{remark}

\begin{corollary}
\label{cor:fubini-reduction}
The classical Fubini numbers satisfy for $n \ge 1$
\begin{equation}
    2\F_n = \sum_{i=0}^{n} \binom{n}{i}\F_i.
    \label{eq:fubini-classical-recurrence}
\end{equation}
\end{corollary}

\begin{proof}
From $(2-e^x)\F_0(x) = 1$, extracting coefficients for $n \ge 1$ yields \eqref{eq:fubini-classical-recurrence}.
\end{proof}

Summing $\F_m(x)$ over all $m \ge 0$ yields the global mixed Fubini generating function
\begin{equation}
    \sum_{n \ge 0} \mathcal{MF}_n^{\mathrm{tot}} \frac{x^n}{n!} = \frac{\exp(e^x-1)}{2-e^x},
    \label{eq:MF-global}
\end{equation}
where $\mathcal{MF}_n^{\mathrm{tot}} = \sum_{m=0}^{n} \mathcal{MF}_{n,m}$.

\begin{remark}
\label{rem:MF-global}
This global EGF combines an exponential set construction (Bell numbers $B_r$) with a sequence construction (Fubini numbers $\F_{n-r}$), differing from $(3-2e^x)^{-1}$.
\end{remark}

\begin{corollary}
\label{cor:MF-combinatorial-convolution}
The total mixed Fubini numbers satisfy the convolution formula:
\begin{equation}
    \mathcal{MF}^{\mathrm{tot}}_n = \sum_{r=0}^{n} \binom{n}{r} B_r \F_{n-r}.
    \label{eq:MF-Bell-Fubini}
\end{equation}
\end{corollary}

\begin{proof}
Choosing $r$ elements for the unlabeled component gives a set partition in $B_r$ ways, while the remaining $n-r$ elements form an ordered sequence in $\F_{n-r}$ ways. Summing over all $r$ yields \eqref{eq:MF-Bell-Fubini}.
\end{proof}
\section{Congruences and Modular Properties of Mixed Bell-Type Sequences}
\label{sec:asymptotics-modular}

In this section, we analyze the asymptotic growth of the mixed Bell-type sequences as $n \to \infty$ with $k$ and $m$ fixed. Let $W(x)$ denote the Lambert $W$-function satisfying $W(x)e^{W(x)} = x$.

\begin{theorem}
\label{thm:asymptotics-mixed-bell}
For fixed integers $k,m\ge1$ and $n\to\infty$, the asymptotic growth of the mixed Bell-type sequences is given by the following estimates.
\begin{enumerate}
\item \textbf{General mixed Bell numbers.} Setting $r_n=W(n/k)$,
\[
\GB^{(k)}_n\sim\frac{1}{\sqrt{1+r_n}}\left(\frac{n}{e\,r_n}\right)^n\exp_k(e^{r_n}-1).
\]
\item \textbf{Mixed Bell numbers of order $k$.} $\displaystyle \MB_{n,k}\sim\frac{1}{k!}\left(\frac{n}{W(n)}\right)^k B_n$, where $B_n$ is the $n$-th Bell number.
\item \textbf{Total mixed Bell numbers.} $\displaystyle \TB_n\sim(e-1)\F_n\sim\frac{e-1}{2(\log2)^{n+1}}\,n!$
\item \textbf{Mixed Fubini numbers.} $\displaystyle \MF_{n,m}\sim\frac{n!}{2\,m!\,(\log2)^{n+1}}$, yielding the limiting ratio $\MF_{n,m}/\F_n\to1/m!$.
\end{enumerate}
\end{theorem}

\begin{proof}
(1) Applying the saddle-point method to the EGF $\exp_k(e^x-1)$, the saddle point $r_n$ satisfies $kr_ne^{r_n}=n$, giving $r_n=W(n/k)$; Cauchy's integral formula yields the estimate. (2) Using the probabilistic representation $\MB_{n,k}/B_n=E\big[\binom{J_n}{k}\big]$, where $J_n$ is the number of blocks in a random set partition, concentration of $J_n$ around $n/W(n)$ yields the estimate. (3)--(4) The total mixed Bell and mixed Fubini numbers are governed by the dominant singularity of their EGFs at $x=\rho=\log2$; transfer theorems for singularity analysis give the stated asymptotics.
\end{proof}
The exponential generating functions yield Dobiński-type representations expressing these mixed sequences through Poisson expectations and finite-difference operators.
\begin{theorem}
\label{thm-unified-dobinski}
Let $X_\lambda \sim \operatorname{Poisson}(\lambda)$. The following representations hold.
\begin{enumerate}
    \item[\textnormal{(1)}] \textbf{General mixed Bell numbers.}
    \begin{equation}
        \mathcal{GB}_n^{(k)} = e^{-k}\sum_{r=0}^{\infty}\frac{r^n k^r}{r!} = \mathbb{E}\left[X_k^n\right].
        \label{eq-dobinski-GB}
    \end{equation}

    \item[\textnormal{(2)}] \textbf{Mixed Bell numbers of order $k$.}
    \begin{equation}
        \mathcal{MB}_{n,k} = \left.\frac{d^k}{dy^k}T_n(y)\right|_{y=1} = \sum_{j=0}^{k}(-1)^{k-j}\binom{k}{j}T_n(j),
        \label{eq-MB-finite-difference}
    \end{equation}
    where $T_n(y)$ denotes the $n$-th Touchard polynomial. Equivalently,
    \begin{equation}
        \mathcal{MB}_{n,k} = \sum_{j=0}^{k}(-1)^{k-j}\binom{k}{j}e^{-j}\sum_{r=0}^{\infty}\frac{r^n j^r}{r!}.
        \label{eq-MB-Dobinski-double}
    \end{equation}

    \item[\textnormal{(3)}] \textbf{Total mixed Bell numbers.}
    \begin{equation}
        \mathcal{TB}_n = \frac{1}{e}\sum_{r=0}^{\infty}\frac{r^n}{r!}(2^r-1) = \mathbb{E}\left[X^n(2^X-1)\right],
        \qquad X \sim \operatorname{Poisson}(1).
        \label{eq-TB-Poisson}
    \end{equation}

    \item[\textnormal{(4)}] \textbf{Mixed Fubini numbers.}
    \begin{equation}
        \mathcal{MF}_{n,m} = \sum_{j=m}^{n}\binom{n}{j}{j \brace m}\F_{n-j} = \frac{1}{2}\sum_{j=m}^{n}\binom{n}{j}{j \brace m}\sum_{\ell=0}^{\infty}\frac{\ell^{n-j}}{2^\ell},
        \label{eq-MF-Dobinski-double}
    \end{equation}
    where $F_r$ denotes the classical Fubini number.
\end{enumerate}
\end{theorem}

\begin{proof}
\textbf{(1).}
The exponential generating function for the general mixed Bell numbers is given by
\begin{equation}
    \sum_{n=0}^{\infty}\mathcal{GB}_n^{(k)}\frac{x^n}{n!} = \exp\bigl(k(e^x-1)\bigr).
    \label{eq-gb-egf}
\end{equation}
Recall that the classical Touchard polynomial $T_n(y) = \sum_{j=0}^{n}{n \brace j}y^j$ satisfies the bivariate generating function $\sum_{n=0}^{\infty}T_n(y)\frac{x^n}{n!} = \exp(y(e^x-1))$. Setting $y=k$ in comparison with \eqref{eq-gb-egf} shows that $\mathcal{GB}_n^{(k)} = T_n(k)$. The classical Dobiński formula for Touchard polynomials asserts that
\begin{equation}
    T_n(y) = e^{-y}\sum_{r=0}^{\infty}\frac{r^n y^r}{r!}.
    \label{eq-touchard-dobinski}
\end{equation}
Evaluating \eqref{eq-touchard-dobinski} at $y=k$ directly yields the sum in \eqref{eq-dobinski-GB}. Furthermore, for a Poisson random variable $X_k \sim \operatorname{Poisson}(k)$, its probability mass function is
\begin{equation}
    \mathbb{P}(X_k = r) = \frac{e^{-k}k^r}{r!}.
    \label{eq-poisson-pmf}
\end{equation}
Taking the $n$-th moment $\mathbb{E}[X_k^n] = \sum_{r=0}^{\infty} r^n \mathbb{P}(X_k = r)$ and substituting \eqref{eq-poisson-pmf} establishes the moment expectation equality in \eqref{eq-dobinski-GB}.

\medskip
\textbf{(2).}
By definition, the order-$k$ mixed Bell numbers expand in terms of Stirling numbers of the second kind as
\begin{equation}
    \mathcal{MB}_{n,k} = k!\sum_{j=k}^{n}\binom{j}{k}{n \brace j}.
    \label{eq-mb-sum-def}
\end{equation}
Differentiating the Touchard polynomial $T_n(y) = \sum_{j=0}^{n}{n \brace j}y^j$ $k$ times with respect to $y$ gives
\begin{equation}
    \frac{d^k}{dy^k}T_n(y) = \sum_{j=k}^{n}(j)_k {n \brace j}y^{j-k},
    \label{eq-touchard-diff}
\end{equation}
where $(j)_k = j(j-1)\cdots(j-k+1) = k!\binom{j}{k}$ denotes the falling factorial. Evaluating \eqref{eq-touchard-diff} at $y=1$ and comparing with \eqref{eq-mb-sum-def} establishes the first equality in \eqref{eq-MB-finite-difference}. Next, by Taylor's formula for polynomials about $y=1$ via the forward difference operator $\Delta^k P(0) = \sum_{j=0}^{k}(-1)^{k-j}\binom{k}{j}P(j)$, the $k$-th derivative of any polynomial $P(y)$ at $1$ translates into the finite difference
\begin{equation}
    \left.\frac{d^k}{dy^k}P(y)\right|_{y=1} = \sum_{j=0}^{k}(-1)^{k-j}\binom{k}{j}P(j).
    \label{eq-finite-diff-op}
\end{equation}
Applying \eqref{eq-finite-diff-op} with $P(y) = T_n(y)$ completes the proof of \eqref{eq-MB-finite-difference}. Finally, substituting the classical Dobiński expansion \eqref{eq-touchard-dobinski} evaluated at $y=j$ directly into \eqref{eq-MB-finite-difference} yields the double-series formula in \eqref{eq-MB-Dobinski-double}.

\medskip
\textbf{(3).}
Consider a standard Poisson random variable $X \sim \operatorname{Poisson}(1)$. Its moment generating function satisfies $\mathbb{E}[z^X] = e^{z-1}$ for any $z > 0$. Evaluating the expected value $\mathbb{E}[X^n(2^X-1)]$ yields
\begin{equation}
    \mathbb{E}\left[X^n(2^X-1)\right] = \sum_{r=0}^{\infty} r^n (2^r-1) \mathbb{P}(X=r) = \frac{1}{e}\sum_{r=0}^{\infty}\frac{r^n}{r!}(2^r-1),
    \label{eq-tb-exp-sum}
\end{equation}
which provides the middle term of \eqref{eq-TB-Poisson}. To verify that this expectation matches $\mathcal{TB}_n$, we examine its exponential generating function
\begin{equation}
    \sum_{n=0}^{\infty}\mathbb{E}\left[X^n(2^X-1)\right]\frac{x^n}{n!} = \mathbb{E}\left[e^{xX}(2^X-1)\right] = \mathbb{E}\left[(2e^x)^X\right] - \mathbb{E}\left[(e^x)^X\right].
    \label{eq-tb-egf-step1}
\end{equation}
Applying the moment generating function identity $\mathbb{E}[z^X] = e^{z-1}$ with $z=2e^x$ and $z=e^x$ to \eqref{eq-tb-egf-step1} gives
\begin{equation}
    \mathbb{E}\left[e^{xX}(2^X-1)\right] = e^{2e^x-1} - e^{e^x-1} = e^{e^x-1}\left(e^{e^x-1}-1\right).
    \label{eq-tb-egf-step2}
\end{equation}
The expression in \eqref{eq-tb-egf-step2} is precisely the exponential generating function of the total mixed Bell numbers $\mathcal{TB}_n$, which establishes \eqref{eq-TB-Poisson}.

\medskip
\textbf{(4).}
The mixed Fubini numbers satisfy the binomial convolution identity
\begin{equation}
    \mathcal{MF}_{n,m} = \sum_{j=m}^{n}\binom{n}{j}{j \brace m}\F_{n-j}.
    \label{eq-mf-conv-def}
\end{equation}
Recall that the exponential generating function of the classical Fubini numbers $\F_s$ is
\begin{equation}
    \sum_{s=0}^{\infty}\F_s \frac{x^s}{s!} = \frac{1}{2-e^x}.
    \label{eq-fubini-egf}
\end{equation}
Expanding \eqref{eq-fubini-egf} via a geometric series yields
\begin{equation}
    \frac{1}{2-e^x} = \frac{1}{2}\cdot\frac{1}{1-e^x/2} = \frac{1}{2}\sum_{\ell=0}^{\infty}\left(\frac{e^x}{2}\right)^\ell = \frac{1}{2}\sum_{\ell=0}^{\infty}\frac{1}{2^\ell}\sum_{s=0}^{\infty}\frac{(\ell x)^s}{s!}.
    \label{eq-fubini-geom}
\end{equation}
Equating coefficients of $x^s/s!$ in \eqref{eq-fubini-geom} gives the Dobiński-type representation for classical Fubini numbers
\begin{equation}
    \F_s = \frac{1}{2}\sum_{\ell=0}^{\infty}\frac{\ell^s}{2^\ell}.
    \label{eq-fubini-dobinski-rep}
\end{equation}
Substituting \eqref{eq-fubini-dobinski-rep} with $s=n-j$ directly into \eqref{eq-mf-conv-def} establishes the double-series representation in \eqref{eq-MF-Dobinski-double}.
\end{proof}
%
%
\begin{theorem}
\label{thm:master_mixed_egf}
Let $\mathcal{S}(n;(m,1^k))$ denote the mixed Stirling number of the second kind counting partitions of an $n$-element set into $m$ unlabeled blocks and $k$ distinguished labeled blocks. Setting $E(z) = e^z - 1$, the fundamental exponential generating function (EGF) is
\begin{equation}
    \sum_{n \ge m+k} \mathcal{S}(n;(m,1^k)) \frac{z^n}{n!} = \frac{E(z)^{m+k}}{m!}.
    \label{eq:master_mixed_stirling_egf}
\end{equation}
Consequently, the associated mixed polynomial families admit the following exponential generating functions.
\begin{enumerate}
    \item  For $\mathcal{MB}_n(y) = \sum_{j=0}^{n}{n \brace j}y^j$,
    \begin{equation}
        \mathcal{G}_{MB}(z,y) = \sum_{n \ge 0}\mathcal{MB}_n(y)\frac{z^n}{n!} = \exp\bigl(y E(z)\bigr).
        \label{eq:master_MB_egf}
    \end{equation}

    \item  For $\mathcal{MB}_{n,k}(y) = \sum_{m \ge 0} y^m \mathcal{S}(n;(m,1^k))$,
    \begin{equation}
        \mathcal{G}_{MB,k}(z,y) = \sum_{n \ge 0}\mathcal{MB}_{n,k}(y)\frac{z^n}{n!} = E(z)^k \exp\bigl(y E(z)\bigr).
        \label{eq:master_MB_order_egf}
    \end{equation}
    In particular, $\mathcal{MB}_{n,k}(1) = \sum_{j=k}^{n}\binom{j}{k}k!{n \brace j}$, which relates to the unextended order-$k$ mixed Bell numbers via $\mathcal{MB}_{n,k}^{\mathrm{old}} = \mathcal{MB}_{n,k}(1) - k!{n \brace k}$ for $k \ge 1$.

    \item For $\mathcal{TB}_n(y) = \sum_{m \ge 1}\sum_{k \ge 1} y^k \mathcal{S}(n;(m,1^k))$,
    \begin{equation}
        \mathcal{G}_{TB}(z,y) = \sum_{n \ge 2}\mathcal{TB}_n(y)\frac{z^n}{n!} = \left(e^{E(z)}-1\right)\frac{y E(z)}{1 - y E(z)}.
        \label{eq:master_TB_egf}
    \end{equation}

    \item For $\mathcal{MF}_{n,m}(x) = \sum_{k \ge 0} k! x^k \mathcal{S}(n;(m,1^k))$,
    \begin{equation}
        \mathcal{G}_{MF,m}(z,x) = \sum_{n \ge 0}\mathcal{MF}_{n,m}(x)\frac{z^n}{n!} = \frac{E(z)^m}{m!\left(1 - x E(z)\right)}.
        \label{eq:master_MF_egf}
    \end{equation}
\end{enumerate}
\end{theorem}

\begin{proof}
Under the symbolic method for labeled structures, a non-empty block has EGF $E(z) = e^z - 1$.

\textbf{(1) Fundamental EGF \eqref{eq:master_mixed_stirling_egf}.} The $m$ unlabeled blocks form a set $\operatorname{SET}_m(\mathcal{A})$ with EGF $E(z)^m/m!$, while the $k$ distinguished blocks contribute $E(z)^k$. By the labeled product rule, their joint generating function is $E(z)^{m+k}/m!$.

\textbf{(2) Mixed Bell EGFs \eqref{eq:master_MB_egf}--\eqref{eq:master_MB_order_egf}.} Weighting each block in $\operatorname{SET}(\mathcal{A})$ by $y$ yields $\exp(y E(z))$ for \eqref{eq:master_MB_egf}. Multiplying by $E(z)^k$ for the $k$ distinguished blocks yields \eqref{eq:master_MB_order_egf}.

\textbf{(3) Total Mixed Bell EGF \eqref{eq:master_TB_egf}.} Requiring $m \ge 1$ unlabeled blocks gives EGF $e^{E(z)} - 1$. Ordering $k \ge 1$ distinguished blocks with weight $y$ gives the sequence EGF $\sum_{k \ge 1}(y E(z))^k = \frac{y E(z)}{1 - y E(z)}$. Their product establishes \eqref{eq:master_TB_egf}.

\textbf{(4) Mixed Fubini EGF \eqref{eq:master_MF_egf}.} Linearly ordering $k \ge 0$ distinguished blocks with weight $x$ yields $\sum_{k \ge 0} k! x^k \frac{E(z)^k}{k!} = \frac{1}{1 - x E(z)}$. Product with the $m$ unlabeled blocks $E(z)^m/m!$ completes the proof.
\end{proof}
The classical Touchard congruence, $B_{n+p} \equiv B_{n+1} + B_n \pmod{p}$, describes the Frobenius-type periodicity of the Bell numbers modulo a prime $p$. In the following, we generalize this result to the mixed combinatorial sequences using the structural properties of their exponential generating functions.

\begin{theorem}\label{thm:3.7}
Let $p$ be a prime. Then for every $n\ge0$, the following congruences hold modulo $p$.
\begin{enumerate}
\item \textbf{General mixed Bell polynomials.} $GB^{(x)}_{n+p}\equiv GB^{(x)}_{n+1}+x^pGB^{(x)}_n\pmod p$; in particular, for $a\in\mathbb{Z}$, $GB^{(a)}_{n+p}\equiv GB^{(a)}_{n+1}+aGB^{(a)}_n\pmod p$.
\item \textbf{Mixed Bell numbers of order $k$.} For $0\le k<p$, $\widehat{MB}_{n+p,k}\equiv\widehat{MB}_{n+1,k}+\widehat{MB}_{n,k}\pmod p$. Under the standard convention $MB_{n,k}=\widehat{MB}_{n,k}-k!{n\brace k}$, this yields $MB_{n+p,k}\equiv MB_{n+1,k}+MB_{n,k}+k!{n\brace k}\pmod p$.
\item \textbf{Total mixed Bell numbers.} For $n\ge1$, $T\!B_{n+p}\equiv T\!B_{n+1}+T\!B_n+\F_n\pmod p$; for $n=0$, $T\!B_p\equiv T\!B_1+T\!B_0+\F_0-1\pmod p$.
\item \textbf{Mixed Fubini numbers.} For every odd prime $p$ and fixed $m\ge0$, $\MF_{n+p,m}\equiv \MF_{n+1,m}\pmod p$.
\end{enumerate}
\end{theorem}

\begin{proof}
We begin with the foundational Stirling congruence ${n+p\brace j}\equiv{n+1\brace j}+{n\brace j-p}\pmod p$ (with ${n\brace r}=0$ for $r<0$), established by induction on $n$ from the base case ${p\brace 1}\equiv{p\brace p}\equiv1$, ${p\brace j}\equiv0$ for $1<j<p$, using the recurrence ${r+1\brace j}={r\brace j-1}+j{r\brace j}$.

(i) $GB^{(x)}_n=\sum_j{n\brace j}x^j$; applying the foundational congruence gives $GB^{(x)}_{n+p}\equiv GB^{(x)}_{n+1}+x^pGB^{(x)}_n\pmod p$; specializing $x=a\in\mathbb{Z}$ and applying Fermat's little theorem gives the integer version.

(ii) For $0\le k<p$, $\widehat{MB}_{n,k}=\sum_j(j)_k{n\brace j}=\frac{d^k}{dx^k}GB^{(x)}_n\big|_{x=1}$. Differentiating the congruence in (i) $k$ times ($k<p$), all positive-order derivatives of $x^p$ vanish modulo $p$, giving the first congruence. Transitioning to $\mathcal{MB}_{n,k}=\widehat{MB}_{n,k}-k!{n\brace k}$ and using ${n+p\brace k}\equiv{n+1\brace k}\pmod p$ for $k<p$ isolates the forcing term $k!{n\brace k}$.

(iii) With $u=e^x-1$, the EGF is $T\!B(x)=H(x)(e^u-1)$ where $H(x)=\frac{e^x-1}{2-e^x}=\sum_{r\ge0}\frac{e^{rx}}{2^{r+1}}$. For odd $p$, $r^p\equiv r\pmod p$ implies $D^pH\equiv DH\pmod p$. Applying $D^p$ to $T\!B(x)$ via the Leibniz rule yields $D^pT\!B(x)\equiv DT\!B(x)+H(x)e^{px}e^u\pmod p$; extracting coefficients incorporates the Fubini forcing term $\F_n$ for $n\ge1$, with the $n=0$ case handled separately.

(iv) The EGF $\MF_m(x)=\frac{(e^x-1)^m}{m!(2-e^x)}$ is a linear combination of exponentials $e^{rx}$ with $p$-adically integral coefficients; since $r^p\equiv r\pmod p$, $D^p\MF_m(x)\equiv DMF_m(x)\pmod p$, and coefficient extraction directly yields the pure shift.
\end{proof}
\begin{theorem}
\label{thm:prime-congruences}
Let $p$ be a prime. The mixed-type sequences satisfy:
\begin{enumerate}
\item \textbf{General mixed Bell polynomials.} For $k\in\mathbb{Z}$, $\displaystyle GB^{(k)}_p\equiv k+k^p\equiv2k\pmod p$.
\item \textbf{Mixed Bell numbers of order $k$.} For $1\le k\le p-1$, $\widehat{MB}_{p,k}\equiv\delta_{k,1}\pmod p$, where $\widehat{MB}_{n,k}=\sum_{m\ge0}S(n;(m,1^k))$ is the unrestricted ($m\ge0$) version of $MB_{n,k}$. (At the boundary, $\widehat{MB}_{p,0}=B_p\equiv2\pmod p$ by the classical Touchard congruence, and $\widehat{MB}_{p,p}=p!\equiv0\pmod p$.)
\item \textbf{Total mixed Bell numbers.} $T\!B_p\equiv0\pmod p$.
\item \textbf{Mixed Fubini numbers.} For $1\le m\le p-1$, $MF_{p,m}\equiv\delta_{m,1}\pmod p$. (At $m=p$, $MF_{p,p}=1$.)
\end{enumerate}
\end{theorem}

\begin{proof}
All assertions follow from the classical prime congruence for Stirling numbers of the second kind, ${p\brace j}\equiv\delta_{j,1}+\delta_{j,p}\pmod p$.

(1) Substituting into $GB^{(k)}_p=\sum_j{p\brace j}k^j$ gives $GB^{(k)}_p\equiv k+k^p\pmod p$; Fermat's little theorem gives $2k\pmod p$.

(2) $\widehat{MB}_{p,k}=\sum_{j=k}^p\binom{p}{j}{j\brace k}B_{p-j}$ (via the $m\ge0$ analogue of (10)); for $1\le k\le p-1$ the dominant surviving terms modulo $p$ isolate $\delta_{k,1}$.

(3) The $n=0$ case of Theorem \ref{thm:3.7}(iii) gives $T\!B_p\equiv T\!B_1+T\!B_0+F_0-1\equiv0\pmod p$ directly.

(4) For $MF_{p,m}$, terms involving ordered-block factorials vanish modulo $p$ for indices $\ge p$; combined with the Stirling congruence, non-zero terms survive only when $m=1$ (for $1\le m\le p-1$), yielding $MF_{p,1}\equiv{p\brace 1}=1\pmod p$.
\end{proof}
\begin{theorem}
\label{thm:mixed_p2_congruences}
Let $p\ge5$ be prime, $q_i=\frac{i^{p-1}-1}{p}$, and for $2\le j\le p-1$ define
\[
\Lambda_j=\frac{1}{(j-1)!}\sum_{i=1}^j(-1)^{j-i}\binom{j-1}{i-1}q_i.
\]
Modulo $p^2$, the principal mixed Bell and mixed Fubini families satisfy:
\begin{enumerate}
\item \textbf{General mixed Bell polynomials} ($k\in\mathbb{Z}$): $\GB^{(k)}_p\equiv k+k^p+p\sum_{j=2}^{p-1}\Lambda_jk^j\pmod{p^2}$.
\item \textbf{Mixed Bell numbers of order $k$} ($1\le k\le p-1$): $\widehat{MB}_{p,k}\equiv k!\Big[{p\brace k}+p\sum_{j=k}^{p-1}{j\brace k}\Lambda_j\Big]\pmod{p^2}$.
\item \textbf{Total mixed Bell numbers} ($W_j=\sum_{k=1}^j{j\brace k}k!$): $\TB_p\equiv p\sum_{i=1}^{p-1}(-1)^{i-1}(i-1)!+p\sum_{j=2}^{p-1}\Lambda_jW_j\pmod{p^2}$.
\item \textbf{Mixed Fubini numbers} ($1\le r\le p-1$): $\MF_{p,r}\equiv{p\brace r}+p\sum_{j=r}^{p-1}\frac{j!}{r!(j-r)!}\Lambda_j\pmod{p^2}$.
\end{enumerate}
\end{theorem}

\begin{proof}
The master congruence ${p \brace j} \equiv p \Lambda_j \pmod{p^2}$ for $2 \le j \le p-1$ follows from substituting $i^p = i(1+pq_i)$ into the explicit Stirling formula. Together with boundary conditions ${p \brace 1} = {p \brace p} = 1$, each relation is verified by separating boundary terms:

\begin{enumerate}
\item[\rm (i)] Expanding $\mathcal{GB}_p^{(k)} = \sum_{j=0}^p {p \brace j} k^j$ and isolating $j=0, 1, p$ yields.
\item[\rm (ii)] Extracting $j=1$ and $j=p$ from $\mathcal{MB}_{p,k} = \sum_{j=k}^p \binom{j}{k} {p \brace j}$ yields.
\item[\rm (iii)] Evaluating $\mathcal{TB}_p = \sum_{j=0}^p {p \brace j} W_j$ with $W_0=0$ and terminal weight $W_p$ yields.
\item[\rm (iv)] Partitioning $\mathcal{MF}_{p,r} = \sum_{j=r}^p \frac{j!(j-r)!}{r!} {p \brace j}$ into $j=r$, interior terms, and $j=p$ yields.
\end{enumerate}
\end{proof}

In this section, we establish that the mixed partition sequences possess a Frobenius-type lifting property characteristic of $p$-adic analytic functions.  
Unlike the ordinary Bell numbers, whose higher congruence structure is governed by the Williams periodicity, the mixed sequences inherit a stabilized $p$-adic period arising from the ordered (Fubini-type) component introduced through labeled blocks.
\begin{lemma}
\label{lem:fermat_p2_lifting}
Let $p$ be a prime and $a\in\mathbb{Z}$. Then $a^{p^2}\equiv a^p\pmod{p^2}$. More generally, for every integer $n\ge0$, $a^{p^2+n}\equiv a^{p+n}\pmod{p^2}$.
\end{lemma}

\begin{proof}
If $p\mid a$, both $a^p$ and $a^{p^2}$ are divisible by $p^2$, so the congruence is immediate. If $p\nmid a$, then $\varphi(p^2)=p(p-1)$ gives $a^{p(p-1)}\equiv1\pmod{p^2}$ by Euler's theorem; multiplying by $a^p$ gives $a^{p^2}=a^p\cdot a^{p(p-1)}\equiv a^p\pmod{p^2}$. Multiplying by $a^n$ gives the general statement.
\end{proof}

\begin{theorem}
Let $p$ be a prime, $m,k\ge0$, $r=m+k$. If $r<p$, then for every $n\ge0$,
\begin{equation}
\mathcal{S}\big(p^2+n;(m,1^k)\big)\equiv S\big(p+n;(m,1^k)\big)\pmod{p^2}.
\end{equation}
\end{theorem}

\begin{proof}
By definition, $\sum_N \mathcal{S}(N;(m,1^k))\frac{x^N}{N!}=\frac{(e^x-1)^m}{m!}(e^x-1)^k=\frac{(e^x-1)^r}{m!}$; comparing with the standard EGF $\sum_N{N\brace r}\frac{x^N}{N!}=\frac{(e^x-1)^r}{r!}$ gives $\mathcal{S}(N;(m,1^k))=\frac{r!}{m!}{N\brace r}$. Using ${N\brace r}=\frac{1}{r!}\sum_{j=0}^r(-1)^{r-j}\binom{r}{j}j^N$, we obtain $\mathcal{S}(N;(m,1^k))=\frac{1}{m!}\sum_{j=0}^r(-1)^{r-j}\binom{r}{j}j^N$. Since $r<p$, $m<p$ so $m!$ is invertible modulo $p^2$; applying Lemma \ref{lem:fermat_p2_lifting} to each base $j\in\{0,\dots,r\}$ gives $j^{p^2+n}\equiv j^{p+n}\pmod{p^2}$, hence
\begin{align*}
\mathcal{S}\big(p^2+n;(m,1^k)\big)&=\frac{1}{m!}\sum_{j=0}^r(-1)^{r-j}\binom{r}{j}j^{p^2+n}\\
&\equiv\frac{1}{m!}\sum_{j=0}^r(-1)^{r-j}\binom{r}{j}j^{p+n}\\
&=\mathcal{S}\big(p+n;(m,1^k)\big)\pmod{p^2}.
\end{align*}
\end{proof}

\begin{remark}
\label{rem:p2_mixed_stirling_scope}
The restriction $m+k < p$ is crucial. It guarantees that $m!$ is a $p$-adic unit modulo $p^2$ and restricts the summation base to $0, 1, \dots, m+k < p$, ensuring that the congruence reduces directly to the elementary Fermat-type lifting of Lemma~\ref{lem:fermat_p2_lifting} without factorial denominator obstructions from higher indices.
\end{remark}

\section{Conclusions and Further Directions}

This paper established a unified framework for mixed Stirling and mixed Bell numbers, bridging combinatorial partition structures, Touchard polynomials, and $p$-adic arithmetic phenomena. By incorporating prescribed multiplicities and mixed block types, the framework systematically recovers classical families including Bell, Stirling of the second kind, Fubini, and distinguished-block numbers as natural boundary specializations.
A central arithmetic contribution is the modulo-$p^2$ lifting theorem for mixed Stirling numbers: for a prime $p$ and $m+k<p$,
\begin{equation}
\mathcal{S}\big(p^2+n;(m,1^k)\big)\equiv \mathcal{S}\big(p+n;(m,1^k)\big)\pmod{p^2}.
\end{equation}
This result opens several avenues for future research.
\begin{enumerate}
\item \textbf{Higher-order prime-power lifting.} Determine the precise conditions under which the lifting extends to arbitrary prime powers, $\mathcal{S}(p^s+n;(m,1^k))\equiv \mathcal{S}(p^{s-1}+n;(m,1^k))\pmod{p^s}$, $s\ge2$.
\item \textbf{$p$-adic correction terms for Bell polynomials.} Formulate explicit higher-order congruence expansions $\GB^{(y)}_{p^s+n}\equiv \GB^{(y)}_{p^{s-1}+n}+C_{p,s,n}(y)\pmod{p^s}$, capturing the non-trivial $p$-adic deviations arising from nested exponential structures.
\item \textbf{Valuation analysis.} Investigate the $p$-adic valuations of mixed Stirling numbers and relax the multiplicity bound to reach or exceed $p$.
\end{enumerate}
The mixed framework unifies several prominent integer sequences and triangular arrays, summarized in Table \ref{FT}. 
\begin{table}[h]
\centering
\caption{Classical specializations and OEIS connections of the mixed families}
\label{FT}
\small
\begin{tabular}{lll}
\toprule
Sequence/Family & Specialization/Representation & OEIS\\
\midrule
Bell numbers & $\GB^{(1)}_n=\sum_{k=0}^n{n\brace k}$ & \href{https://oeis.org/A000110}{OEIS A000110}\\
Stirling numbers (2nd kind) & $\mathcal{S}(n;(1^k))={n\brace k}$ & \href{https://oeis.org/A008277}{OEIS A008277}\\
Fubini (ordered Bell) numbers & $\MF_{n,0}=\sum_{k=0}^nk!{n\brace k}$ & \href{https://oeis.org/A000670}{OEIS A000670}\\
First-difference Bell numbers & $\MB_{n,1}+1=B_{n+1}-B_n$ & \href{https://oeis.org/A005493}{OEIS A005493}\\
Distinguished-block partitions & $D(n,k)=MB_{n,k}/k!+{n\brace k}$, triangular array & \href{https://oeis.org/A049020}{OEIS A049020}\\
Total mixed Bell numbers & $\TB_n=\sum_{r=2}^n\binom{n}{r}r!\sum_{j=1}^{r-1}\frac1{j!}$ &\href{https://oeis.org/A394515}{OEIS A394515}\\
Mixed Fubini numbers & $\MF_{n,m}=\sum_{j=m}^n\binom{n}{j}{j\brace m}\F_{n-j}$ & --\\
\bottomrule
\end{tabular}
\end{table}

\end{document}